\documentclass{amsart}
\usepackage[utf8]{inputenc}

\usepackage{amsmath, amsthm, amssymb, amsfonts, verbatim}
\usepackage{graphicx}
\usepackage{array}
\usepackage[bookmarks]{hyperref}
\usepackage[active]{srcltx}
\usepackage{amscd}
\usepackage{epsfig}
\usepackage{color}
\usepackage{comment}
\usepackage{setspace}
\usepackage{textcomp}
\usepackage{stmaryrd}
\usepackage{multirow}
\usepackage{amsaddr}

\newcommand{\pmat}[1]{\begin{pmatrix} #1 \end{pmatrix}}

\newcommand{\beq} {\begin{equation}}
\newcommand{\eeq} {\end{equation}}
\newcommand{\bdm} {\begin{displaymath}}
\newcommand{\edm} {\end{displaymath}}

\newcommand{\bit}{\begin{itemize}}
\newcommand{\eit}{\end{itemize}}
\newcommand{\bde}{\begin{description}}
\newcommand{\ede}{\end{description}}

\newcommand{\ben}{\begin{enumerate}}
\newcommand{\een}{\end{enumerate}}
\newcommand{\algn}[1]{\begin{align} #1 \end{align}}
\newcommand{\algns}[1]{\begin{align*} #1 \end{align*}}
\newcommand{\mltln}[1]{\begin{multline} #1 \end{multline}}
\newcommand{\mltlns}[1]{\begin{multline*} #1 \end{multline*}}

\newcommand{\gats}[1]{\begin{gather*} #1 \end{gather*}}
\newcommand{\barr}{\begin{array}}
\newcommand{\earr}{\end{array}}
\newcommand{\subeqns}[2]{\begin{subequations} \label{#1} \algn{#2} \end{subequations}}

\newcommand{\half} {\ensuremath{\frac{1}{2}}}
\newcommand{\mc}[1]{\mathcal{#1}}
\newcommand{\LRp}[1]{\left( #1 \right)}

\newcommand{\LRa}[1]{\left< #1 \right>}

\newcommand{\LRc}[1]{\left\{ #1 \right\}}
\newcommand{\jump}[1] {\ensuremath{\left[\!\left[{#1}\right]\!\right]}}

\newcommand{\avg}[1] {\ensuremath{\left\{\!\!\left\{{#1}\right\}\!\!\right\}}}

\newcommand{\ra}{\rightarrow}

\newcommand{\e}{\epsilon}

\newcommand{\R}{{\mathbb R}}

\newcommand{\X}{{\mathbb X}}
\newcommand{\I}{\Bbb{I}}

\renewcommand{\div}{\operatorname{div}}

\newcommand{\esssup}{\operatorname{esssup}}

\newcommand{\bs}{\boldsymbol}

\newcommand{\pd}{\partial}
\newcommand{\nor}[1]{\left\Vert#1\right\Vert}
\newcommand{\norw}[2]{\left\Vert#1\right\Vert_{#2}}

\newcommand{\n}{\bs{n}}

\newtheorem{theorem}{Theorem}[section]
\newtheorem{lemma}[theorem]{Lemma}

\newtheorem{rmk}[theorem]{Remark}

\newcommand{\deui}{\dot{e}_{\bs{u}}^I}
\newcommand{\deuh}{\dot{e}_{\bs{u}}^h}

\newcommand{\depfi}{\dot{e}_{\pf}^I}

\newcommand{\dpf}{\dot{p}}
\newcommand{\dpfh}{\dot{p}_{h}}
\newcommand{\ddt}{\frac{d}{dt}}

\newcommand{\dt}{\Delta t}

\newcommand{\epf}{e_{\pf}}
\newcommand{\epfi}{e_{\pf}^I}
\newcommand{\epfh}{e_{\pf}^h}

\newcommand{\eu}{e_{\bs{u}}}

\newcommand{\euh}{e_{\bs{u}}^h}

\newcommand{\fb}{\bs{f}}

\newcommand{\kapb}{\ul{\bs{\kappa}}}

\newcommand{\laminv}{\lambda^{-1}}

\newcommand{\pf}{p}

\newcommand{\qf}{q}
\newcommand{\Qf}{Q}
\newcommand{\Qh}{Q_h}

\newcommand{\ub}{\bs{u}}
\newcommand{\ubh}{\bs{u}_{h}}

\newcommand{\ul}{\underline}

\newcommand{\vb}{\bs{v}}
\newcommand{\Vb}{\bs{V}}
\newcommand{\Vbh}{\bs{V}_{h}}

\newcommand{\wb}{\bs{w}}

\begin{document}

\title[Coupling nonconforming and EG methods for poroelasticity]{Coupling nonconforming and enriched Galerkin methods for robust discretization and fast solvers of poroelasticity problems }
%
\author{Jeonghun J. Lee$^1$, Jacob Moore$^2$ }

\address{$^{1,2}$ Department of Mathematics, Baylor University, Waco, Texas, USA  }

\email{$^1$jeonghun\_lee@baylor.edu, $^2$jacob\_moore5@baylor.edu}

\subjclass[2000]{Primary: 65N12, 65N15}

\begin{abstract}
	In this paper we propose a new finite element discretization for the two-field formulation of poroelasticity which uses the elastic displacement and the pore pressure as primary variables. The main goal is to develop a numerical method with small problem sizes which still achieve key features such as parameter-robustness, local mass conservation, and robust preconditionor construction. For this we combine a nonconforming finite element and the interior over-stabilized enriched Galerkin methods with a suitable stabilization term. Robust a priori error estimates and parameter-robust preconditioner construction are proved, and numerical results illustrate our theoretical findings.
\end{abstract}

\keywords{Biot's consolidation model, error analysis, preconditioning}
\date{April, 2023}

\maketitle

\section{Introduction}

The poroelasticity theory is a theory on fluid flows in deformable porous media. The interactive behavior of fluid flows and elastic deformation results in a system of partial differential equations coupling elasticity and porous media flow equations.
Since the first poroelasticity models were derived by Biot \cite{Biot-low-frequency,Biot-high-frequency}, mathematical properties and numerical methods for the models have been widely explored. We refer to \cite{Zenisek:1984,Showalter:2000} for well-posedness of linear quasi-static poroelasticity problems. For numerical methods there are a number of previous results with various formulations and finite element spaces \cite{Santos:1986a,Santos:1986b,Zienkiewicz-Shiomi:1984,Murad:1996,Phillips:2009, Phillips:2008,Phillips:2009, Yi:2013,Yi:2014}. 
The main focus of more recent studies is developing 
numerical methods which provide good features of numerical solutions. The features of interest are robustness of numerical solutions for nearly incompressible solid matrix and saturating fluids, local mass conservation, and fast solvers which show robust performance for realistic physical parameter ranges. Here we list some recent results
\cite{Yi:2017,JJLee:2016,Lee:2017,Rodrigo:2016,Hu:2017,Boffi:2016,JJLee:2017,Oyarzua:2016,Feng:2018,Kanschat:2018,Hong:2018,Riviere:2017,Fu:2018-biot,Kraus:2021,Botti:2021,Lee-Yi:2023} but we admit that this list of references is by no means complete.

The goal of this paper is to develop a numerical method for poroelasticity with small problem sizes such that numerical solutions satisfy the three features: 1) the error bounds are robust for a nearly incompressible elastic matrix, for arbitrarily small constrained specific storage coefficients, for time interval, 2) local mass conservation of numerical solutions holds, 3) preconditioners robust for mesh refinements as well as the physical parameter values are available.
For this goal we use a formulation which has the displacement of poroelastic matrix $\ub$ and the fluid pressure $p$ as unknowns. 

We remark that there are a couple of previous works on robust numerical methods using the two-field formulation. In \cite{Fu:2018-biot}, primal hybridized discontinuous Galerkin methods are used for the displacement and pressure, and it is proved that the methods give robust numerical solutions for nearly incompressible materials. In \cite{Lee-Yi:2023}, the enriched Galerkin method for linear elasticty  in \cite{Yi-Lee-Zikatanov:2022} and the enriched Galerkin method for the Poisson equations in \cite{Lee-Lee-Wheeler-16} are coupled. To the best of our knowledge, construction of parameter-robust preconditioning for these methods are not available. In contrast, we show that the numerical method in this paper provides parameter-robust preconditioner construction in addition to robust error estimates and local mass conservation. 

For finite element discretization of the two variables we use the Mardal--Tai--Winther (MTW) element, a nonconforming $H^1$ finite element developed in \cite{Mardal-Tai-Winther:2002} for $\ub$ and the interior over-stabilized enriched Galerkin method for the pore pressure $p$ in \cite{Lee-Ghattas}. For this combination of discretization methods, we prove the a priori error estimates for semidiscrete solutions without Gr\"onwall inequalities and show a preconditioner construction guidance with the operator preconditioning approach in \cite{Mardal-Winther:2011}.

The paper is organized as follows. In Section~\ref{sec:prelim} we introduce symbols, definitions, and the partial differential equation formulation of Biot's consolidation model. In Section~\ref{sec:discretization}, finite element methods for spatial discretization and auxiliary results for the stability of discretization are presented. The a priori error analysis of semidiscrete solutions is proved in Section~\ref{sec:error-analysis}. An abstract form of parameter-robust preconditioners  is proposed in Section~\ref{sec:preconditioning}, and numerical results for convergence of errors and for efficiency of preconditioners are presented in Section~\ref{sec:numerics}.
Conclusive remarks are given in Section~\ref{sec:conclusion}. 

\section{Preliminaries} \label{sec:prelim}

\subsection{Notations}
Let $\Omega$ be a bounded domain in $\R^n$, $n=2,3$ with a polygonal or polyhedral boundary. For a nonnegative real number $m$, $H^m (\Omega)$ and $H^m(\Omega; \R^n)$ denote the standard $\R$ and $\R^n$-valued Sobolev spaces based on the $L^2$ norm. For a Banach space $\mc{X}$ and $(a, b) \subset \R$, $C^0 (a, b; \mc{X})$ is the set of functions $f : (a, b) \ra \mc{X}$ which are continuous in $t \in (a,b)$. For an integer $m \geq 1$ we define 
\begin{align*}
C^m (a, b ; \mc{X}) := \{ f \, | \, \pd^{i}f/\pd t^{i} \in C^0(a, b\,;\mc{X}), \, 0 \leq i \leq m \},
\end{align*}
where $\pd f/\pd t$ is the Fr\'echet derivative of $f$ (see e.g., \cite{Yosida-book}). Recall the Bochner spaces by the norm
\begin{align*}
\| f \|_{L^p(a,b; \mc{X})} = 
\begin{cases}
\left( \int_a^b \| f \|_{\mc{X}}^p ds \right)^{1/p}, \quad 1 \leq p < \infty, \\
\esssup_{t \in (a,b)} \| f \|_{\mc{X}}, \quad p = \infty.
\end{cases}
\end{align*}
%
If a time interval $J \subset \R $ is clear in context, then we use $L^p \mc{X}$ to denote $L^p(J; \mc{X})$ for simplicity. 
We define the Bochner-Sobolev spaces $W^{k,p}(J; \mc{X})$ for nonnegative integer $k$ and $1 \leq p \leq \infty$ as the closure of $C^k (J; \mc{X})$ with the norm $\| f \|_{W^{k,p} \mc{X}} = \sum_{i=0}^k \| \pd^i f / \pd t^i \|_{L^p \mc{X}}$. 
We adopt the convention $\| f, g \|_\mc{X} = \| f \|_\mc{X} + \| g \|_\mc{X}$ for simplicity, and $\dot{u}$, $\ddot{u}$ is used to denote $\partial u / \partial t$, $\partial^2 u/ \partial t^2$. 

In this paper we use $\mc{T}_h$ to denote a shape-regular triangulation of $\Omega$ for which $h$ is the maximum diameter of triangles/tetrahedra and $\mathcal{E}_h$ is the corresponding set of edges/faces. For a set $G$ with a positive $n$-dimensional Lebesgue measure and for functions $f, g \in L^2(\Omega)$ we define $\LRp{ f, {g} }_G$ by
\algns{
\LRp{ {f}, {g} }_G := \int_G {f} \cdot {g} \,dx .
}
For a finite-dimensional vector space $\X$ on $\R$ and functions $f, g \in L^2(\Omega;\X)$, $\LRp{ {f}, {g} }_G$ is similarly defined using the natural inner product on $\X$. 
For $e \in \mathcal{E}_h$ and functions ${f}, {g} \in L^2(e)$ we define
\algns{
\langle {f}, {g} \rangle_e = \int_e {f} \cdot {g} \,ds, \qquad \langle {f}, {g} \rangle = \sum_{e \in \mathcal{E}_h} \langle {f}, {g} \rangle_e. 
}
$\LRa{f, g}_e$ and $\LRa{f, g}$ are defined similarly for $f, g \in L^2(e; \X)$. 

For an integer $k \geq 0$ and for each $K \in \mc{T}_h$, $\mc{P}_k(K)$ is the space of polynomials of degree $\le k$ on $K$, and $\mc{P}_k(\mc{T}_h)$ denotes the space 
\algns{
\mc{P}_k(\mc{T}_h) = 
  \LRc{q \in L^2(\Omega) \;:\; q|_K \in \mc{P}_k(T), \; K \in \mc{T}_h } .
}
For a vector space $\X$, we use $\mathcal{P}_k(G; \X)$ and $\mathcal{P}_k(\mathcal{T}_h; \X)$ to denote the space of $\X$-valued polynomials with same conditions. 

For a nonnegative function (or a positive semidefinite tensor) $w$ on $\Omega$, $\nor{q }_{0,w}$ denotes $\nor{q }_{0,w} = \LRp{w q, q}_{\Omega}^{1/2}$.

\subsection{The Biot's consolidation model}
Throughout this paper we limit our interest on quasistatic consolidation problems.
In our formulation of the model, $\bs{u}$ is the displacement of porous media, $p$ is the pore pressure, $\bs{f}$ is the body force, $g$ is the mass change rate of fluid. The governing equations of the quasistatic consolidation model with an isotropic elastic porous medium are 
\subeqns{eq:strong-eqs}{
\label{eq:strong-eq1}-\div \LRp{ 2 \mu \e(\bs{u}) + (\lambda \div \ub - \alpha p) \I } &= \bs{f}, \\
\label{eq:strong-eq2} s_0 \dot{p} + \alpha \div \dot{\bs{u}} - \div (\ul{\bs{\kappa}} \nabla p) &= g, 
}
where $\mu$ and $\lambda$ are the Lam\'e coefficients, $s_0 \geq 0 $ is the constrained specific storage coefficient, $\underline{\bs{\kappa}}$ is the hydraulic conductivity tensor, $\alpha>0$ is the Biot--Willis coefficient which is close to 1, and $\Bbb{I}$ is the identity matrix. 

We assume that $\mu$ is uniformly bounded above and below with positive constants. We assume $\lambda$ has a uniformly positive lower bound but $\lambda$ may not have a uniform upper bound and $\lambda = +\infty$ corresponds to the incompressibility of the solid matrix. 
We assume that there are constants $c_0, c_1$ such that 
\algns{
  0 \le c_0 \le s_0 (x) \le c_1 , \qquad x \in \Omega.
}
We remark that $s_0$ is related to $\alpha$, the porosity $\phi$, and the bulk moduli of the solid and fluid. 
Under the assumption that $\phi$ is uniform with $0 < \phi < \alpha$, if the solid is not incompressible, then $s_0 \ge C/\lambda$ holds with a constant $C$ of scale 1. However, $s_0$ may vanish on a subdomain if $\lambda = + \infty$ on the subdomain and the fluid is incompressible. The hydraulic conductivity tensor $\ul{\bs{\kappa}} = \ul{\bs{\kappa}}(x)$ is positive definite with uniform lower and upper bounds $\kappa_0, \kappa_1 >0$, i.e., 
\algns{
\kappa_{0} | \xi |^2 \le \xi^T \ul{\bs{\kappa}}(x) \xi  \le \kappa_{1} | \xi |^2 , \qquad \forall \;0 \not = \xi \in \R^n,\quad  \text{a.e.} \; x \in \Omega .
}
On details of deriving these equations from physical modelling, we refer to standard references on poroelasticity theory, e.g., \cite{Wang-poroelasticity-book}. 

For well-posedness of the problem, the equations \eqref{eq:strong-eqs} need appropriate boundary and initial conditions. We assume that there are partitions of $\pd \Omega$ which are 
\begin{align*}
\pd \Omega = \Gamma_p \cup \Gamma_f, \qquad \pd \Omega = \Gamma_d \cup \Gamma_t, \qquad | \Gamma_d |, |\Gamma_p| > 0
\end{align*}
where $| \Gamma |$ is the $(n-1)$-dimensional Lebesgue measure of $\Gamma$. We also assume that boundary conditions are given as
\begin{align}
\label{eq:bc}  p(t) = 0 \text{ on } \Gamma_p, \quad - \ul{\bs{\kappa}} \nabla p(t) \cdot \bs{n} = 0 \text{ on } \Gamma_f, \quad \bs{u}(t) = 0 \text{ on } \Gamma_d, \quad \underline{\bs{\sigma}}(t) \bs{n} = 0 \text{ on } \Gamma_t,
\end{align}
for all $t \in (0, T]$ where $\bs{n}$ is the outward unit normal vector field on $\pd \Omega$ and $\underline{\bs{\sigma}} := 2 \mu \e(\bs{u}) + (\lambda \div \ub - \alpha p) \I$, the Cauchy stress tensor.
Here we only consider the homogeneous boundary condition for simplicity but our method can be easily extended to problems with nonhomogeneous boundary conditions. We also assume that given initial data $p(0), \bs{u}(0)$ and $\bs{f}(0)$ satisfy the compatibility condition \eqref{eq:strong-eq1}.
Well-posedness of this system under these assumptions can be found in \cite{Showalter2000}. 



\section{Finite element discretization} \label{sec:discretization}
In this section we present a variational formulation and finite element discretization of \eqref{eq:strong-eqs}, and prove auxiliary results for the stability and error analyses.

\subsection{Variational formulation of continuous problem}
Let us define function spaces 
\algns{ 
\Vb = \LRc{ \vb \in {H}^1 (\Omega; \R^n) \;:\; \vb|_{\Gamma_d} = 0 }, \quad 
\Qf = \{ q \in H^1(\Omega) \;:\; q|_{\Gamma_p} = 0 \} .
}
Then, by the integration by parts, we can show that a solution of \eqref{eq:strong-eqs}, $(\bs{u}(t), p(t))$, satisfies the following variational formulation: 
\subeqns{eq:weak-up-eqs}{
\label{eq:weak-up-eq1} \LRp{2 \mu \e(\bs{u}(t)), \e(\bs{v}) }_{\Omega} + \LRp{ \lambda \div \bs{u}(t), \div \bs{v}}_{\Omega} - \LRp{ \alpha \pf(t) , \div \bs{v} }_{\Omega} &= \LRp{ \bs{f}(t), \bs{v} }_{\Omega},  \\ 
\label{eq:weak-up-eq2} - \LRp{ \alpha \div \dot{\bs{u}}(t) , \qf }_{\Omega} - \LRp{ s_0 \dpf(t) , \qf }_{\Omega} - \LRp{ \kapb \nabla \pf(t) , \nabla \qf }_{\Omega}  &= -\LRp{ g(t), \qf }_{\Omega} 
}
for all $(\bs{v}, \qf) \in \Vb \times \Qf$ and for all $t \ge 0$. 

For well-posedness we need initial conditions $(\bs{u}(0), \pf(0)) \in \Vb \times \Qf$. By the compatibility assumption,  
%
\algn{ \label{eq:comp-cond} 
\LRp{2 \mu \e(\bs{u} (0)), \e(\bs{v}) }_{\Omega} + \LRp{ \lambda \div \bs{u}(0), \div \bs{v}}_{\Omega} - \LRp{ \alpha \pf (0), \div \bs{v} }_{\Omega} &= \LRp{ \bs{f} (0), \bs{v} }_{\Omega}  
}
for all $\vb \in \Vb$.

In summary the variational problem of \eqref{eq:strong-eqs} is the following: For given initial condition $(\bs{u}(0), \pf(0)) \in \Vb \times \Qf$ satisfying \eqref{eq:comp-cond} find $(\ub, \pf) \in C^1(0,T; \Vb) \times C^1(0,T; \Qf)$ such that \eqref{eq:weak-up-eqs} holds for all $t \ge 0$.

\subsection{Finite elements for spatial discretization}
In this subsection we introduce finite element methods for spatial discretizations of $\ub$ and $\pf$.

For discretization of $\ub$ we use the vector-valued nonconforming $H^1$ elements by Mardal, Tai, Winther (\cite{Mardal-Tai-Winther:2002,Tai-Winther:2006}). Here we only summarize key features of the finite elements, and we refer to the original articles for detailed definitions of the finite elements. 

Let $\Vb_h$ be the subspace of the Mardal-Tai-Winther element with an additional condition that  
\begin{align} \label{eq:sigma-property5}
	\text{all DOFs associated to the edges/faces in } \Gamma_d \text{ vanish}.
\end{align}
By definition one can see that $\Vb_h \subset H(\div, \Omega)$ and $\div \Vbh = \mc{P}_0(\mc{T}_h)$ but $\Vb_h \not \subset H^1(\Omega; \R^n)$. If we define $\| \vb \|_{1,h}$ by $\| \bs{v} \|_{1,h}^2 := \| \nabla \vb \|_{0}^2 + \| \vb \|_0^2$ with $\nabla$ and $\e$, by the element-wise gradient and symmetric gradient operators, then a discrete Korn inequality 
\algns{
  \| \bs{v} \|_{1,h} \ge C \| \e(\bs{v}) \|_{0}
}
holds for $C>0$ independent of $h$ (see~\cite{Mardal-Winther:2006}). In addition, an inf-sup condition
\algn{ \label{eq:MTW-inf-sup}
  \inf_{q \in \mc{P}_0(\mc{T}_h)} \sup_{\vb \in \Vbh} \frac{\LRp{ \div \vb, q}_{\Omega} }{ \| \vb \|_{1,h} \| q \|_0 } \ge C >0
}
holds with $C$ independent of $h$.

For discretization of $\pf$ we use the interior over-stabilized enriched Galerkin methods developed in \cite{Lee-Ghattas}. We first define $\Qh^c$, $\Qh^0$, $\Qh$ by
\algns{
  \Qh^c := \mc{P}_1(\mc{T}_h) \cap H^1(\Omega), \quad \Qh^0 := \mc{P}_0(\mc{T}_h) , \quad \Qh := \Qh^c + \Qh^0 .
}

%

Let use define $a_{\ub}: (\Vb + \Vb_h) \times (\Vb + \Vb_h) \rightarrow \mathbb{R}$ and $\| \cdot \|_{\Vb_h}$ by 
\algn{
\label{eq:a_u-form} a_{\ub} \LRp{ \vb, \tilde{\vb}} &:= \LRp{2 \mu \e(\vb), \e(\tilde{\vb}) } + \LRp{ \lambda \div \vb, \div \tilde{\vb}} , \qquad \| \vb \|_{\Vb_h} = \LRp{a_{\ub}\LRp{\vb, \vb} }^{1/2} . 
}
%
To define the bilinear form for $\pf$ with the interior over-stabilized enriched Galerkin method, we define the jump and average operators for functions which have well-defined traces on edges/faces. More specifically, for $e \in \mc{E}_h^{\pd}$, we define
\algns{
\jump{q}|_e = q|_e \n , \quad \avg{q}|_e = q|_e, \qquad \jump{\bs{q}}|_e = \bs{q}|_e \cdot \n, \quad \avg{\bs{q}}|_e = \bs{q}|_e 
}
where $q$ and $\bs{q}$ are scalar and $\R^n$-vector-valued functions such that their traces on edge/face $e$ are well-defined as integrable functions, and $\n$ is the outward unit normal vector field on $e$. For $e \in \mc{E}_h^0$, $T_+$ and $T_-$ are the two elements such that $e = \pd T_+ \cap \pd T_-$. If $q$ is a scalar function on $T_+ \cup T_-$, then we use $q^{\pm}$ and $\nabla q^{\pm}$ to denote the restrictions of $q$ and $\nabla q$ on $T_+$ and $T_-$. We use $\n^{\pm}$ to denote the unit outward normal vector fields of $T_\pm$ and we again omit the restriction $|_e$ if it is clear in context. Then the jumps and the averages of $q$ and $\bs{q}$ on $e \in \mc{E}_h^i$ are defined by 
\algns{
\jump{q}|_e &= q^+ \n^+ + q^- \n^-, & \avg{q}|_e &= \half (q^+ + q^- ), \\
\jump{\bs{q}}|_e &= \bs{q}^+ \cdot \n^+ + \bs{q}^- \cdot \n^- , & \avg{\bs{q}}|_e &= \half (\bs{q}^+ + \bs{q}^- ) .
}
If $e \in \mc{E}_h$ is clear in context, we will use $\jump{q}$ instead of $\jump{q}|_e$ for the jump of $q$ on $e \in \mc{E}_h$. The same simplification will apply to other quantities.

For $\qf \in \Qh$ we use $\qf^c$ and $\qf^0$ to denote the components in $\Qh^c$ and $\Qh^0$ which are unique up to a constant addition. For $\qf, \tilde{\qf} \in \Qh$ we now define
\mltln{ 
\label{eq:a_p-form} a_p \LRp{\qf, \tilde{\qf}} := \LRp{\kapb \nabla \qf^c, \nabla \tilde{\qf^c}}_{\Omega} - \langle \kapb \avg{\nabla \qf^c}, \jump{\tilde{\qf}^0} \rangle - \langle \jump{\qf^c}, \avg{\kapb \nabla \tilde{\qf}} \rangle \\ 
 + \gamma \kapb_{avg} \sum_{e \in \mc{E}_h^0} \langle h_e^{-1-\beta} \jump{\qf^0}, \jump{\tilde{q}^0} \rangle_e + \gamma \kapb_{avg} \sum_{e \in \mc{E}_h^{\pd}} \langle h_e^{-1} (q^c + q^0), (\tilde{q}^c + \tilde{q}^0) \rangle_e 
}
with $\beta \ge 1$ and a sufficiently large $\gamma >0$. This bilinear form is well-defined independent of the choices of $\qf^c, \tilde{\qf}^c$, $\qf^0$, $\tilde{\qf}^0$ up to constant addition because all terms except the last term in this bilinear form vanishes if $\qf^c, \tilde{\qf}^c$, $\qf^0$, $\tilde{\qf}^0$ are constant functions on $\Omega$, and the last term is independent of the decompositions of $q^c, q^0, \tilde{q}^c, \tilde{q}^0$. Note that there exists a unique decomposition $q = q^c + q^0 \in \Qh^c + \Qh^0$ for $q \in \Qh$ such that $q^0\in \Qh^0$ is mean-value zero on $\Omega$. We will use this decomposition frequently in the rest of this paper for analysis.
If $\gamma>0$ is sufficiently large depending on the shape regularity of the triangulation $\mc{T}_h$, then $a_p(\cdot, \cdot)$ is a coercive bilinear form on $\Qh$. 
Note that \eqref{eq:a_p-form} is the enriched Galerkin (EG) method for second order elliptic equations in \cite{Lee-Lee-Wheeler-16} if $\beta = 0$. In this paper we assume $\beta \ge 1$ which is the interior over-stabilized enriched Galerkin methods (IOS-EG) in \cite{Lee-Ghattas}. 
Convergence rates of numerical solutions are same in EG and IOS-EG methods.
However, as is discussed in \cite{Lee-Ghattas}, IOS-EG methods provide more robust fast solver algorithms. 
In this paper we will also show this by numerical experiments later. 

We also define a stabilization bilinear form $S : \Qh \times \Qh \ra \R$ by 
\algn{
  S(q, \tilde{q}) = \gamma \LRp{ \laminv \sum_{e \in \mc{E}_h^0} \langle h_e^{-1} \jump{\qf}, \jump{\tilde{q}} \rangle_e + \sum_{e \in \mc{E}_h^0} \langle C_1 h_e^{-1} \jump{\qf}, \jump{\tilde{q}} \rangle_e }
}
with a sufficiently large $\gamma >0$. This stabilization term will be used to discretize the time derivative term of $\pf$ in order to guarantee the stability of discretization independent of the lower bound of $s_0 \ge 0$ and robust preconditioning. Note that the edge/face integrals in $S(\cdot, \cdot)$ are defined only on the interior edges/faces $\mc{E}_h^0$, so only mean-value zero components of $q^0$ and $\tilde{q}^0$ contribute to $S(q, \tilde{q})$. We will use this fact in our stability analysis later. 

Let us define the norm $\| \vb \|_{\Vbh}$ for $\vb \in \Vb_h$ by 
\algns{
  \| \vb \|_{\Vbh}^2 := \int_{\Omega} \LRp{ 2 \mu \epsilon(\vb) : \epsilon(\vb) + \lambda (\div \vb)^2 } \, dx 
}
where $\e(\vb) :\e(\vb)$ is the Frogenius inner product on matrices. For $q \in \Qh$ we define $\| q \|_{0,\laminv}$ and $\| q \|_{h}$ by 
\algn{
  \label{eq:q-L2-norm} \| q \|_{0,\laminv}^2 &:=   \int_{\Omega} \laminv q^2 \, dx , \\
  \label{eq:q-nonstandard-norm} \| q \|_{h}^2 &:= \int_{\Omega} \kapb \nabla q^c \cdot \nabla q^c \, dx + \sum_{e \in \mc{E}_h^{0} } \gamma \kapb_{avg} h_e^{-1-\beta} \int_e \jump{q}^2 \, ds \\
  \notag &\quad + \sum_{e \in \mc{E}_h^{\pd} } \gamma \kapb_{avg} h_e^{-1} \int_e \jump{q^c}^2 \, ds + \sum_{e \in \mc{E}_h^{\pd} } \gamma \kapb_{avg} h_e^{-1} \int_e \jump{q^0}^2 \, ds.
}
In the rest of this subsection we introduce a few results for stability of the discretization. These results will be used to prove well-posedness as well as the a priori error analysis of the discretized system.

We begin with defining an auxiliary finite element space $R_h$ and a linear operator $I_h: \Qh^c \ra R_h$. For a vertex $v_i$ of $\mc{T}_h$ let $\chi_i$ be the characteristic function on the support of the shape function of an element in $\Qh^c$ associated to the vertex degree of freedom at $v_i$. Let $R_h$ be the subspace of $\Qh^0$ defined by 
\algns{ 
  R_h = \{ q^0 \in \Qh^0 \,:\, q^0 = \sum_{i=1}^N c_i \chi_i, c_i \in \R \} 
}
where $N$ is the number of vertices of $\mc{T}_h$. We define $I_h : \Qh^c \ra R_h$ by 
\algns{
  I_h q^c = \sum_{i=1}^N q^c(v_i) \chi_i .
}
The following result can be obtained with the same argument in \cite[Lemma~1]{Lamichhane:2014}. 
\begin{lemma} \label{lemma:integral-identity}
  Given $\vb \in \Vbh$ and $q^c \in \Qh^c$ 
  \algn{ \label{eq:integral-identity}
    (n+1) \LRp{ \div \vb, q^c }_{\Omega} = \LRp{ \div \vb, I_h q^c }_{\Omega} .
  }
\end{lemma}
\begin{proof}
Since $\div \Vbh = \Qh^0$ we can see that \eqref{eq:integral-identity} holds for any nodal basis $q^c \in \Qh^c$, so the conclusion follows.
\end{proof}
The following norm equivalence result is also proved in \cite{Lamichhane:2014}. The proof follows from a standard scaling argument, so we state the result here without proof.
\begin{lemma} \label{lemma:Ih-equivalence}
  There exist $C_1, C_2>0$ independent of $h$ such that 
  \algns{
    C_1 \| q^c \|_0 \le \| I_h q^c \|_0 \le C_2 \| q^c \|_0, \quad q^c \in \Qh^c.
  }
\end{lemma}
By combining Lemma~\ref{lemma:integral-identity} and Lemma~\ref{lemma:Ih-equivalence}, we can prove an inf-sup condition with $\lambda$-weighted norms.
\begin{lemma} \label{lemma:aux-inf-sup}
  There exist $C>0$ independent of $h$ such that 
  \algns{
    \inf_{q^c \in \Qh^c} \sup_{\vb \in \Vbh} \frac{\LRp{ \div \vb, q^c }_{\Omega} }{\| \vb \|_{\Vbh} \| q^c \|_{0,\laminv} } \ge C .
  }
\end{lemma}
\begin{proof}
The inf-sup condition
  \algns{
    \inf_{q^c \in \Qh^c} \sup_{\vb \in \Vbh} \frac{\LRp{ \div \vb, q^c }_{\Omega} }{\| \vb \|_{1,h} \| q^c \|_0 } \ge C 
  }
follows immediately by Lemma~\ref{lemma:integral-identity} and Lemma~\ref{lemma:Ih-equivalence}. As an equivalent condition of this inf-sup condition, for any given $q^c \in \Qh^c$ there exists $\wb \in \Vbh$ such that 
\algns{
\LRp{\div \bs{w}, \tilde{q}^c }_{\Omega} = \LRp{q^c, \tilde{q}^c}_{\Omega} \quad \forall \tilde{q}^c \in \Qh^c, \qquad \| \bs{w} \|_{1,h} \le C_1 \|q^c \|_{0}
}
with $C_1 >0$ independent of $h$. We define $\vb = \lambda^{-1} \wb$. Then, 
\algn{ \label{eq:tmp-identity}
\LRp{ \div \vb, q^c}_{\Omega} = \int_{\Omega} \laminv (q^c)^2 \,dx = \| q^c \|_{0,\laminv}^2 .
}
In addition, since $\mu \le C \lambda$ with a uniform constant $C>0$, we can obtain that 
\algns{
  \| \vb \|_{\Vbh}^2 &= \int_{\Omega} \LRp{ 2 \mu \e(\vb) : \e(\vb) + \lambda (\div \vb)^2 } \,dx \\
  &= \lambda^{-2} \int_{\Omega} \LRp{ 2 \mu \e(\wb) : \e(\wb) + \lambda (\div \wb)^2 } \,dx \\
  &\le \LRp{2 C_{\mu} + 1} \laminv \| \wb \|_{1,h}^2 \\
  &\le \LRp{2 C_{\mu} + 1} C_1 \| q^c \|_{0,\laminv}^2 .
}
Therefore, the assertion holds by \eqref{eq:tmp-identity} and the above estimate.
\end{proof}
By a discrete Poincare inequality in \cite{Brenner:2003} we can obtain the following lemma.
\begin{lemma} \label{lemma:discrete-poincare}
  There exists $C_P>0$ independent of $h$ such that
  \algn{
    \| \qf^0 \|_{0,s_0} + \| \qf^0 \|_{0,\laminv} \le C_P S(\qf^0, \qf^0)^{\frac 12} .
  }
  holds for all $\qf^0 \in \Qh^0$ with $\int_{\Omega} \qf^0 \,dx = 0$. 
\end{lemma}

We introduce another interpolation operator $I_h^c: \Qh^0 \ra \Qh^c$ for analysis. 

\begin{lemma} \label{lemma:S-inequality}
There exists $I_h^c: \Qh^0 \to \Qh^c$, a linear operator, such that
\algn{
\|q^0 - I_h^c q^0 \|_{0,\laminv} 
\le C S \LRp{q, q}^{\half}
}
with $C>0$ independent of $h$.
\end{lemma}
\begin{proof}
We define $I_h^c$ by the reconstruction operator in \cite{Buffa-Ortner:2009} and consider $I_h^c q^0$ for $q^0 \in \Qh^0$.
For a simplex $K$ in $\mc{T}_h$ let $M_K$ be the union of all (closed) simplices in $\mc{T}_h$ adjacent to $K$ and $\text{int} (M_K)$ is the interior of $M_K$. It is known that 
\algns{
  \int_{T} (q^0 - I_h^c q^0)^2 \,dx \le C h_K^2 \sum_{e \subset \text{int} (M_K)} h_e^{-1} \int_e \jump{q^0}|_e^2 \,ds 
}
with $C$ independent of the mesh sizes \cite[Theorem~6]{Buffa-Ortner:2009}. By the shape regularity assumption $h_K \sim h_e$ in the above inequality with constants independent of the mesh sizes,  the conclusion follows because $\{M_K\}_{K \in \mc{T}_h}$ covers $\Omega$ with only uniformly bounded number of finite overlapping.
\end{proof}

\begin{lemma} \label{lemma:weak-inf-sup}
There exists $C_1, C_2 > 0$ independent of $h$ satisfying the following: For any given $q \in \Qh$ one can find $\vb \not = 0 \in \Vbh$ satisfying
\algn{
\frac{\LRp{ \div \vb, q }_{\Omega} }{\|\vb\|_{\Vb_h}} \ge C_1 \|q\|_{0,\lambda^{-1}} - C_2 S \LRp{q, q}^{\half} .
}
\end{lemma}
\begin{proof}
Let $q = q^c + q^0 \in \Qh^c + \Qh^0$ be given and consider the decomposition $q = q^c + I_h^c q^0 - I_h^c q^0 + q^0$. By the inf-sup condition of $\Vbh$ and $\Qh^c$ in Lemma~\ref{lemma:aux-inf-sup} and a scaling with $\lambda$, there exists $\bs{w} \in \Vbh$ such that
\algn{
\LRp{\div \bs{w}, q^c + I_h^c q^0}_{\Omega} &= \|q^c + I_h^c q^0 \|_{0, \lambda^{-1}}^2 , \\
\| \bs{w} \|_{\Vbh} &\le C_3 \|q^c + I_h^c q^0 \|_{0,\lambda^{-1}}
}
with a constant $C_3$ independent of $h$. By the triangle inequality 
\algn{ \label{eq:S-triangle-ineq}
\|q^c + I_h^c q^0 \|_{0,\lambda^{-1}} &= \| \qf + (I_h^c \qf^0 - \qf^0) \|_{0,\laminv} \\
&\le \|q\|_{0,\lambda^{-1}} + \|I_h^c q^0 - q^c \|_{0, \lambda^{-1}} . 
}
By the definition of $\bs{w}$ and the Cauchy-Schwarz inequality,
\algns{
\LRp{\div \bs{w}, q}_{\Omega} &= \LRp{\div \bs{w}, q^c + I_h^c q^0 - I_h^c q^0 + q^0}_{\Omega} \\ \notag
&= \| q^c + I_h^c q^0 \|_{\lambda^{-1}}^2 - \LRp{\div \bs{w}, I_h^c q^0 - q^0}_{\Omega} \\ \notag
&\ge \|q^c + I_h^c q^0 \|_{\lambda^{-1}}^2 - \| \div \bs{w} \|_{0,\lambda} \| I_h^c q^0 - q^0 \|_{0, \laminv}
}
Since $\| \div \wb \|_{0,\lambda} \le \| \wb \|_{\Vbh} \le C_3 \| q^c + I_h^c q^0 \|_{0,\laminv}$, dividing by $\| \bs{w} \|_{\Vbh}$ gives
\algns{
\frac{\LRp{ \div \wb, q}_{\Omega}}{\| \wb \|_{\Vbh} } &\ge \frac{1}{C_3} \|q^c + I_h^c q^0 \|_{0, \laminv} - \| I_h^c q^0 - q^0 \|_{0, \laminv}  .
}
By \eqref{eq:S-triangle-ineq} and Lemma~\ref{lemma:S-inequality}, we get
\algns{
\frac{\LRp{ \div \wb, q}_{\Omega}}{\| \wb \|_{\Vbh} } &\ge \frac{1}{C_3} \| q \|_{0, \laminv} - \frac{C_3 + 1}{C_3} \| I_h^c q^0 - q^0 \|_{0, \laminv}  \\ \notag
&\ge \frac{1}{C_3} \| q \|_{0, \laminv}  - \frac{ C_4 (C_3 +1) }{C_3} S \LRp{q, q}^{\half} .
}
The conclusion holds with $C_1 = 1/C_3$ and $C_2 = C_4(C_3+1)/C_3$. 
\end{proof}

Let $B : (\Vbh \times \Qh) \times (\Vbh \times \Qh) \ra \R$ be defined by 
\algns{
  B((\vb, \qf), (\tilde{\vb}, \tilde{\qf})) &= a_{\ub} (\vb, \tilde{\vb}) - \LRp{ \alpha \qf, \div \tilde{\vb} }_{\Omega} - \LRp{ \alpha \tilde{\qf}, \div \vb }_{\Omega} - \LRp{ s_0 \qf, \tilde{\qf}}_{\Omega} - S\LRp{ \qf, \tilde{\qf}} .
}
For simplicity we define $\| (\vb, q) \|_{\Vbh \times \Qh}$ by 
\algns{
  \| (\vb, \qf) \|_{\Vbh \times \Qh} := \LRp{ \| \vb \|_{\Vbh}^2 + \| q \|_{0,\laminv}^2 + \| q \|_{0,s_0}^2 }^{\frac 12}.
}
Then, we prove the following stability result.
\begin{lemma} \label{lemma:aux-inf-sup}
  There exists $C>0$ independent of $h$ such that
  \algns{
  \inf_{(\tilde{\vb}, \tilde{\qf}) \in \Vbh \times \Qh} \sup_{(\vb, \qf) \in \Vbh \times \Qh} \frac{B((\vb, \qf), (\tilde{\vb}, \tilde{\qf})) }{ \| (\bs{v}, q) \|_{\Vbh \times \Qh} \| (\tilde{\vb}, \tilde{q}) \|_{\Vbh \times \Qh}  } \ge C > 0 .
  }
\end{lemma}
\begin{proof}
For given $(\tilde{\vb}, \tilde{\qf}) \in \Vbh \times \Qh$ it suffices to show that there exists $(\vb, \qf) \in \Vbh \times \Qh$ such that 
\algns{
  \| (\vb, \qf) \|_{\Vbh \times \Qh} &\le C_3  \| (\tilde{\vb}, \tilde{q}) \|_{\Vbh \times \Qh} , \\
  B((\vb, \qf), (\tilde{\vb}, \tilde{\qf})) &\ge C_4  \| (\tilde{\vb}, \tilde{q}) \|_{\Vbh \times \Qh}^2
}
for some constants $C_3, C_4 >0$ independent of $h$. 

Suppose that $(0,0)\not = (\tilde{\vb}, \tilde{\qf}) \in \Vbh \times \Qh$ is given. Let $\wb$ be an element in $\Vbh$ such that 
\algn{ \label{eq:w-weak-inf-sup}
  \frac{\LRp{\div \wb, \tilde{q}}_{\Omega} }{\| \wb \|_{\Vbh} } \ge C_1 \| \tilde{q} \|_{0,\laminv} - C_2 S(\tilde{q}, \tilde{q})^{\frac 12} 
}
by Lemma~\ref{lemma:weak-inf-sup}. By rescaling we can assume that $\| \wb \|_{\Vbh} = \| q \|_{0,\laminv}$ without loss of generality, so we can obtain 
\algn{ \label{eq:w-weak-inf-sup2}
  {\LRp{\div \wb, \tilde{q}}_{\Omega} } \ge C_1' \| \tilde{q} \|_{0,\laminv}^2 - C_2' S(\tilde{q}, \tilde{q})
}
from \eqref{eq:w-weak-inf-sup} with Young's inequality. 

Take $(\vb, q)$ as $\vb = \tilde{\vb} + \delta_1 \wb$ and $q = - \tilde{q}$ with $\delta_1>0$ which will be determined later. Then, 
\mltln{ \label{eq:B-intm}
  B((\vb, q), (\tilde{\vb}, \tilde{q})) \\
  = a_{\ub} (\tilde{\vb}, \tilde{\vb}) + \delta_1 a_{\ub} (\wb, \tilde{\vb}) + \delta_1 \LRp{ \alpha \tilde{q}, \div \wb }_{\Omega} + \LRp{ s_0 \tilde{q}, \tilde{q}}_{\Omega} + S(\tilde{q}, \tilde{q}) .
}
By \eqref{eq:w-weak-inf-sup} with $\| \wb \|_{\Vbh}=1$ and the assumption that $\alpha$ has a uniformly positive lower bound, we have
\algns{
\LRp{ \alpha \tilde{q}, \div \wb}_{\Omega} \ge C_1'' \| \tilde{q} \|_{0,\laminv}^2 - C_2'' S(\tilde{q}, \tilde{q})
}
with $C_1'', C_2''$ depending on $C_1'$, $C_2'$ and the lower bound of $\alpha$. Note that 
\algns{
  \delta_1 | a_{\ub} (\wb, \tilde{\vb}) | \le \frac 12 \| \tilde{\vb} \|_{\Vbh}^2 + \frac{\delta_1^2}{2} \| \wb \|_{\Vbh}^2 = \frac 12 \| \tilde{\vb} \|_{\Vbh}^2 + \frac{\delta_1^2}{2} \| \tilde{q} \|_{0,\laminv}^2   .
}
If we use these two inequalities to \eqref{eq:B-intm}, then we get
\algns{
  B((\vb, q), (\tilde{\vb}, \tilde{q})) &\ge a_{\ub} (\tilde{\vb}, \tilde{\vb}) - \delta_1 |a_{\ub} (\wb, \tilde{\vb})| + \delta_1 \LRp{ C_1'' \| \tilde{q} \|_{0,\laminv}^2 - C_2'' S(\tilde{q}, \tilde{q}) } \\
  &\quad + \LRp{ s_0 \tilde{q}, \tilde{q}}_{\Omega} +  S(\tilde{q}, \tilde{q}) \\
  &\ge \frac 12 a_{\ub} (\tilde{\vb}, \tilde{\vb}) - \frac{\delta_1^2}2 \| \tilde{q} \|_{0,\laminv}^2 + \delta_1 \LRp{ C_1'' \| \tilde{q} \|_{0,\laminv}^2 - C_2'' S(\tilde{q}, \tilde{q}) } \\
  &\quad +  \LRp{ s_0 \tilde{q}, \tilde{q}}_{\Omega} +  S(\tilde{q}, \tilde{q}) .
}
If we choose a sufficiently small $\delta_1$ satisfying $\delta_1 C_1'' \ge \delta_1^2$ and $\delta_1 C_2'' \le 1/2$, then there exists $C_4>0$ depending only on $C_1''$, $C_2''$ such that 
\algns{
  B((\vb, q), (\tilde{\vb}, \tilde{q})) &\ge C_4 \| \delta_1 (\tilde{\vb}, \tilde{q}) \|_{\Vbh \times \Qh}^2 .
}
By the definition of $(\vb, q)$ it is not difficult to see that 
$\| (\vb, \qf) \|_{\Vbh \times \Qh} \le C_3  \| (\tilde{\vb}, \tilde{q}) \|_{\Vbh \times \Qh}$ with a constant $C_3$ depending only on $\delta_1$. Therefore, the assertion is proved.
\end{proof}
Before we prove the a priori error analysis, we show that the proposed method needs considerably smaller number of degrees of freedom (DOFs).
For the discretization of $\ub$ the MTW element has minimal inter-element continuity for discrete Korn inequality (\cite{Mardal-Winther:2006}), so we only need to compare the degrees of freedom for $p$ when $p$ is discretized by the dual form of mixed method with the lowest order Raviart-Thomas element. 
\begin{rmk}
	For easy computation of DOFs we consider a triangulation of the unit cube domain $\Omega$ such that $\Omega$ is subdivided into $N \times N \times N$ subcubes and each subcube is decomposed to 6 tetrahedra. Since each subcube has 6 internal faces, the numbers of vertices, faces, and tetrahedra are $(N+1)^3$, $6 N^2(N+1) + 6N^3$, $6N^3$. The total number of DOFs of mixed method discretization of $p$ with the lowest order Raviart-Thomas element, is $6 N^2(N+1) + 6N^3 + 6N^3 \sim 18 N^3$. After hybridization (see, e.g.,~\cite{Arnold:1985}), only one face DOF is needed, so this can be reduced to $\sim 12 N^3$. The lowest order enriched Galerkin method needs only one DOF for vertices and one DOF for tetrahedra, so the total number of DOFs is $(N+1)^3 + 6 N^3 \sim 7 N^3$ which still is considerably smaller than the one of hybridized mixed method. For hexahedral finite elements the numbers of vertices, faces, cubes are $(N+1)^3$, $3N^2(N+1)$, $N^3$. Therefore, the hybridized method with the lowest order Raviart-Thomas element needs approximately $3N^3$ DOFs whereas the IOS-EG method needs approximately $2N^3$ DOFs.
\end{rmk}

\subsection{Semidiscrete solutions and well-posedness}

According to the variational formulation \eqref{eq:weak-up-eqs} we define the semidiscrete problem as follows: For compatible numerical initial data $(\ub_h(0), \pf_h(0)) \in \Vb_h \times \Qh$ satisfying 
\algn{ \label{eq:numerical-comp-data}
  a_{\ub} \LRp{ \ub_h(0), \vb } - \LRp{ \alpha \pf_h(0), \div \vb }_{\Omega} = \LRp{\fb(0), \vb}_{\Omega} \qquad \forall \vb \in \Vbh
}
we find a semidiscrete solution $(\ub_h, p_h) : C^1(0,\infty; \Vb_h \times \Qh)$ such that
\subeqns{eq:semidiscrete-eqs}{
\label{eq:semidiscrete-eq1} a_{\ub} \LRp{\ub_h (t), \vb} - \LRp{ \alpha p_h (t), \div \vb }_{\Omega} &= \LRp{ \bs{f} (t), \vb }_{\Omega} , \\
\label{eq:semidiscrete-eq2} - \LRp{ \alpha \div \dot{\ub}_h (t) , \qf }_{\Omega} - S\LRp{\dpf_h (t), \qf } - \LRp{ s_0 \dpf_h (t) , \qf }_{\Omega} - a_p \LRp{ \pf_h (t) , \qf }  &= -\LRp{ g (t), \qf}_{\Omega} 
}
for all $(\vb, \qf) \in \Vbh \times \Qh$ and for all $t \ge 0$.
\begin{theorem}
Suppose that $(\ub_h(0), \pf_h(0)) \in \Vbh \times \Qh$ satisfies \eqref{eq:numerical-comp-data}. Then, \eqref{eq:semidiscrete-eqs} has a unique solution. 
\end{theorem}
\begin{proof}
For well-posedness of \eqref{eq:semidiscrete-eqs} we consider a matrix realization of \eqref{eq:semidiscrete-eqs}. We identify $\Vbh$ and $\Qh$ to $\R^{N_{\ub}}$ and $\R^{N_{\pf}}$ with certain global bases of $\Vbh$ and $\Qh$, and denote the vector realizations of $\ub_h$, $\pf_h$ and the $L^2$ projections of $\bs{f}, g$ to $\Vbh, \Qh$ by $\vec{\ub}_h$, $\vec{\pf}_h$, $\vec{\bs{f}}$, $\vec{g}$. Let 
\algns{
  \Bbb{A}_{\ub}, \quad \Bbb{B}, \quad \Bbb{S}, \quad \Bbb{A}_{\pf}
}
be the matrices generated by the bilinear forms 
\algns{
  a_{\ub}(\vb, \tilde{\vb}), \quad - \LRp{\alpha \div \vb , \qf}_{\Omega}, \quad S(\qf, \tilde{\qf}) + \LRp{s_0 \qf, \tilde{\qf} }_{\Omega}, \quad a_p(\qf, \tilde{\qf}) 
}
for $\vb, \tilde{\vb} \in \Vbh$, $\qf, \tilde{\qf} \in \Qh$ with the identifications $\Vbh \leftrightarrow \R^{N_{\ub}}$ and $\Qh \leftrightarrow \R^{N_{\pf}}$. Then, \eqref{eq:semidiscrete-eqs} in matrix form is 
\algns{
  \Bbb{A}_{\ub} {\vec{\ub}}_h(t) + \Bbb{B}^T \vec{\pf}_h(t) &= \vec{\bs{f}(t)}, \\
  \Bbb{B} \dot{\vec{\ub}}_h(t) - \Bbb{S} \dot{\vec{\pf}}_h(t) - \Bbb{A}_{\pf} \vec{\pf}_h(t) &= - \vec{g}(t) .
}
Since $\Bbb{A}_{\ub}$ is invertible $\vec{\ub}_h = - \Bbb{A}_{\ub}^{-1} \Bbb{B}^T \vec{\pf}_h + \Bbb{A}_{\ub}^{-1} \vec{\bs{f}}$ from the the first equation. If we use this to the second equation, then we get 
\algns{
  - \LRp{\Bbb{B} \Bbb{A}_{\ub}^{-1} \Bbb{B} + \Bbb{S} } \dot{\vec{\pf}}_h(t) - \Bbb{A}_{\pf} \vec{\pf}_h(t) = - \Bbb{B} \Bbb{A}_{\ub}^{-1} \dot{\vec{\bs{f}}}(t) - \vec{g}(t) .
}
This is a standard form of system of linear ordinary differential equations, and this equation has a unique solution if $\Bbb{B} \Bbb{A}_{\ub}^{-1} \Bbb{B} + \Bbb{S}$ is invertible. By Lemma~\ref{lemma:aux-inf-sup} the matrix 
\algns{
 \pmat{\Bbb{A}_{\ub} & \Bbb{B}^T \\ \Bbb{B} & -\Bbb{S} }
}
is invertible, so its Schur complement $-(\Bbb{B} \Bbb{A}_{\ub}^{-1} \Bbb{B} + \Bbb{S})$ is invertible as well. Therefore, \eqref{eq:semidiscrete-eqs} has a unique solution.
\end{proof}

\section{Error analysis of semidiscrete solutions}
\label{sec:error-analysis}

For error analysis we define $\Pi_h^{\Vb} : \Vb \to \Vb_h$ and $\Pi_h^{\Qf} : \Qf \to \Qh$ by 
%
\algns{
a_{\ub} \LRp{\Pi_h^{\Vb} \vb, \tilde{\vb}} &= a_{\ub} \LRp{\vb, \tilde{\vb}}  & &  \forall \tilde{\vb} \in \Vb_h , \\
a_p \LRp{\Pi_h^{\Qf} \qf, \tilde{\qf}} &= a_p \LRp{  \qf,  \tilde{\qf}} & &  \forall \tilde{\qf} \in \Qh .
}
It is known that 
\subeqns{eq:intp-ineq}{
  \label{eq:intp-u-ineq} \| \vb - \Pi_h^{\Vb} \vb \|_{a_{\ub}} \le Ch^m \| \vb \|_{m+1} , \\
  \label{eq:intp-p-ineq} \| q - \Pi_h^Q q \|_0 \le C \| \qf - \Pi_h^Q \qf \|_{h} \le Ch^m \| \qf \|_{m+1}, \\
  \label{eq:intp-jump-ineq} \sum_{e \in \mc{E}_h^0} h_e^{-1-\beta} \int_e \jump{\Pi_h^Q \qf }^2 \,ds \le C h^{2m} \| \qf \|_{m+1}^2 
}
in \cite{Mardal-Tai-Winther:2002,Tai-Winther:2006} and \cite{Lee-Ghattas}. Moreover, if $\Omega$ satisfies the full elliptic regularity assumption, then 
\algns{
  \| \qf - \Pi_h^Q \qf \|_0 \le Ch^{m+1} \| \qf \|_{m+1} .
}
\begin{theorem}
Suppose that $(\ub, \pf)$ and $(\ub_h, \pf_h)$ are the solutions of \eqref{eq:weak-up-eqs} and \eqref{eq:semidiscrete-eqs} with compatible numerical initial data $(\ub_h(0), \pf_h(0)) \in \Vb_h \times \Qh$ satisfying \eqref{eq:numerical-comp-data} and
\subeqns{eq:init-cond-approx}{
  2 \mu \| \ub(0) - \ub_h (0) \|_0^2 + \lambda \| \div (\ub(0) - \ub_h (0)) \|_0^2 &\le Ch^m \| \ub(0) \|_{m+1},  \\
  \| \pf(0) - \pf_h(0) \|_0, \| \pf(0) - \pf_h(0) \|_{1,h} &\le Ch^m \| \pf(0) \|_{m+1} .
}
Then, 
\algn{
  \notag &\| \ub(t) - \ub_h(t) \|_{\Vbh} + \| \pf(t) - \pf_h(t) \|_{0, s_0} \\
  \label{eq:up-error} &\le Ch^m \LRp{ \| \ub(0) \|_{m+1} + \| \pf(0) \|_{m+1} } + C \lambda^{-\half} h^m \| \pf \|_{W^{1,1}(0,t; H^{m+1})} \\
  \notag &\quad + Ch^m \| \ub \|_{W^{1,1}(0,t; H^{m+1})} 
}
and 
\mltln{ \label{eq:pH1-error}
  \| \pf(t) - \pf_h(t) \|_{1,h} \\
  \le Ch^m \LRp{\| p(0) \|_{m+1} + \lambda^{-\half} \| p \|_{W^{1,2}(0,t; H^{m+1})} + \| \ub \|_{W^{1,2}(0,t; H^{m+1})} } .
}
\end{theorem}
\begin{proof}

Let $e_{\ub} (t):= \ub (t) - \ub_h(t)$ and $e_{\pf}(t) := \pf (t) - \pf_h (t)$. 
Now, subtracting \eqref{eq:semidiscrete-eqs} from \eqref{eq:weak-up-eqs}, we can derive our error equations as 
\subeqns{eq:error-form}{
\label{eq:error-form1} a_{\ub} \LRp{e_{\ub} (t), \vb} - \LRp{\alpha e_{\pf} (t), \div \vb}_{\Omega} &= 0 , \\
\label{eq:error-form2} - \LRp{\alpha \div \dot{e}_{\ub} (t), \qf}_{\Omega} - S\LRp{\dot{e}_{\pf}(t), q } - \LRp{s_0 \dot{e}_{\pf} (t), \qf}_{\Omega} + a_p \LRp{\epf (t), \qf} &= 0  
}
for all $\vb_h \in \Vb_h$, $\qf \in \Qh$. Defining 
\algns{
e_{\ub}^h (t) &:= \Pi_h^{\Vb} \ub (t) - \ub_h (t), & e_{\ub}^I (t) &:= \ub (t) - \Pi_h^{\Vb} \ub (t) , \\
e_{\pf}^h (t) &:= \Pi_h^{\Qf} \pf (t) - \pf_h (t), & e_{\pf}^I (t) &:= \pf (t) - \Pi_h^{\Qf} \pf (t) ,
}
one can see that $e_{\ub}(t) = e_{\ub}^I (t) + e_{\ub}^h (t)$, 
%
and $a_{\ub} \LRp{e_{\ub}^I (t), \vb} = 0$ for all $\vb \in \Vb_h$ by the definition of $\Pi_h^{\Vb}$. 
%
Thus, we can rewrite \eqref{eq:error-form} as 
\subeqns{eq:new-error-form}{
\label{eq:new-error-form1} a_{\ub} \LRp{e_{\ub}^h (t), \vb} - \LRp{\alpha \LRp{e_{\pf}^h (t) + e_{\pf}^I (t)}, \div \vb}_{\Omega} &= 0 , \\
\notag - \LRp{\alpha \div(\dot{e}_{\ub}^h (t) + \dot{e}_{\ub}^I (t)), \qf}_{\Omega} - S(\dot{e}_{\pf}^h (t) + \dot{e}_{\pf}^I(t) , q) & \\
\label{eq:new-error-form2} - \LRp{s_0 (\dot{e}_{\pf}^h (t) + \dot{e}_{\pf}^I (t)), \qf}_{\Omega} - a_p \LRp{e_{\pf}^h (t), \qf} &= 0 
}
for $(\vb, \qf) \in \Vb_h \times \Qf_h$. 
For our error analysis take $\vb = \dot{e}_{\ub}^h (t)$ and $\qf = -e_{\pf}^h (t)$ and add \eqref{eq:new-error-form1} and \eqref{eq:new-error-form2} to obtain the following:
\mltln{
\label{eq:error1} a_{\ub} \LRp{e_{\ub}^h (t), \dot{e}_{\ub}^h (t)} - \LRp{\alpha e_{\pf}^I (t), \div \dot{e}_{\ub}^h (t)}_{\Omega} \\
+ \LRp{\alpha \div \dot{e}_{\ub}^I (t), e_{\pf}^h (t)}_{\Omega} + S\LRp{\dot{e}_{\pf}^h (t) + \dot{e}_{\pf}^I (t), e_{\pf}^h (t)} \\
+ \LRp{s_0 (\dot{e}_{\pf}^h (t) + \dot{e}_{\pf}^I (t)), e_{\pf}^h (t)}_{\Omega} + a_p \LRp{e_{\pf}^h (t), e_{\pf}^h (t)} = 0 .
}
We can rewrite this as 
\mltln{
\label{eq:error2} \half \ddt a_{\ub} \LRp{\euh (t), \euh (t)} + \half \ddt \LRp{s_0 \epfh (t), \epfh (t)}_{\Omega} \\
+ a_p \LRp{\epfh (t), \epfh (t)}  + \half \ddt S\LRp{ \epfh(t), \epfh(t)} \\
= \LRp{\alpha \epfi (t), \div \dot{e}_{\ub}^h (t)}_{\Omega} - \LRp{\alpha \div \dot{e}_{\ub}^I (t), \epfh (t)}_{\Omega} - S\LRp{ \depfi(t), \epfh(t)} .
}
Let us define $X(t) \ge 0$ by 
\algns{
X(t)^2 := a_{\ub} \LRp{\euh (t), \euh (t)} + (s_0 \epfh (t), \epfh (t))_{\Omega} + S \LRp{\epfh (t), \epfh (t)}. 
}
For the proof of \eqref{eq:up-error} we first prove  
\mltln{
\label{eq:error-bound1} X(t) \le X(0) + C \lambda^{-\half} \max_{0 \le s \le t} \| \epfi (s) \|_0  \\ 
 + C \LRp{ \lambda^{-\half} \int_0^t \| \depfi (s) \|_0 \, ds + \int_0^t \| \div \deui (s) \|_0 \, ds} 
}
for any $t >0$. 

To prove it, we integrate \eqref{eq:error2} from 0 to $t$ and get
\mltln{
\label{eq:error3} \half X(t)^2 - \half X(0)^2 + \int_0^t a_p \LRp{\epfh (s), \epfh (s)} \, ds \\
= \int_0^t  \LRp{\alpha \epfi (s), \div \deuh (s)}_{\Omega} \, ds - \int_0^t \LRp{\alpha \div \depfi (s), \epfh (s)}_{\Omega} \, ds  - S\LRp{ \depfi(t), \epfh(t)} .
}
By the integration by parts in $t$, 
\algns{
\int_0^t \LRp{\alpha \epfi (s), \div \deuh (s)}_{\Omega} \, ds &=  \LRp{\alpha \epfi (t), \div \euh (t)}_{\Omega} - \LRp{\alpha \epfi (0), \div \euh (0)}_{\Omega} \\
&\quad - \int_0^t \LRp{\alpha \depfi (s), \div \euh (s)}_{\Omega} \, ds .
}
Plugging this into \eqref{eq:error3} yields
\algn{
\notag &\half X(t)^2 + \int_0^T a_p \LRp{\epfh (s), \epfh (s)} \, ds \\ 
\label{eq:error4} &= \half X(0)^2 + \LRp{\alpha \epfi (t), \div \euh (t)}_{\Omega} - \LRp{\alpha \epfi (0), \div \euh (0)}_{\Omega} \\
\notag &\quad - \int_0^T \LRp{\alpha \depfi (s), \div \euh (s)}_{\Omega} \, ds - \int_0^T \LRp{\alpha \div \depfi (s), \epfh (s)}_{\Omega} \, ds \\
\notag &\quad - S\LRp{ \depfi(t), \epfh(t)} .
}
We now prove \eqref{eq:error-bound1} assuming that
\algn{ \label{eq:Xmax-assumption}
X(t) =  \max_{0 \le s \le t} X(s) .
}
%
%
%
With this assumption we can bound some terms in \eqref{eq:error4} as follows:
\algn{
\label{eq:bound1} \left|\LRp{\alpha \epfi (t), \div \euh (t)}_{\Omega}\right| &\le C \lambda^{-\half} \| \epfi (t) \|_0 X(t) , \\
\label{eq:bound2} \left|\LRp{\alpha \epfi (0), \div \euh (0)}_{\Omega}\right| &\le C \lambda^{-\half} \| \epfi (0) \|_0 X(0) , \\
\label{eq:bound3} \left| \int_0^t \LRp{\alpha \depfi (s), \div \euh (s)}_{\Omega} \, ds \right| &\le C \lambda^{-\half} \int_0^t \| \depfi (s) \|_0 \, ds X(t) , \\
\label{eq:bound4} \left| \int_0^t \LRp{\alpha \div \depfi (s), \epfh (s)}_{\Omega} \, ds \right| &\le C \int_0^T \| \div \deui (s) \|_0 \, ds X(t) .
}
Plugging these into \eqref{eq:error4}, we derive 
\mltlns{ 
	X(t)^2 + 2 \int_0^t a_p \LRp{\epfh (s), \epfh (s)} \, ds \le X(0)^2 + C \lambda^{-\half} \| \epfi (0) \|_0 X(0) \\ 
	+ C \LRp{\lambda^{-\half} \| \epfi (t) \|_0 + \lambda^{-\half} \int_0^t \| \depfi (s) \|_0 \, ds + \int_0^T \| \div \deui (s) \|_0 \, ds} X(t) .
}
If we divide by $X(t)$, then we can obtain \eqref{eq:error-bound1}. For general $t$ such that \eqref{eq:Xmax-assumption} is not true, note that there exists $0 \le t_M<t$ such that 
\algns{
 X(t_M) = \max_{0 \le s \le t_M} X(s) , \qquad X(t) < X(t_M) . 
}
By the same argument as before for $t_M$ we can obtain 
\mltln{
\label{eq:error-bound2} X(t_M) \le X(0) + C \lambda^{-\half} \| \epfi (0) \|_0  \\ 
 + C \LRp{\lambda^{-\half} \| \epfi (t_M) \|_0 + \lambda^{-\half} \int_0^{t_M} \| \depfi (s) \|_0 \, ds + \int_0^{t_M} \| \div \deui (s) \|_0 \, ds} .
}
We can now derive \eqref{eq:error-bound1} by  
\algns{
X(t) < X(t_M), \qquad \max_{0 \le s \le t_M} \| \epfi (s) \|_0 \le \max_{0 \le s \le t} \| \epfi (t) \|_0, 
}
and
\mltlns{
  \lambda^{-1} \int_0^{t_M} \| \depfi (s) \|_0 + \int_0^{t_M} \| \div \deui (s) \|_0 \, ds \\
  \le \lambda^{-1} \int_0^{t} \| \depfi (s) \|_0 + \int_0^{t} \| \div \deui (s) \|_0 \, ds .
}
Then, \eqref{eq:up-error} follows by the triangle inequality, \eqref{eq:error-bound1}, \eqref{eq:init-cond-approx}, and the interpolation error bounds \eqref{eq:intp-ineq}. 
\end{proof}

\section{Preconditioners for fast solver algorithms}
\label{sec:preconditioning}
In this section we present an abstract form of preconditioners for our discretization of poroelasticity equations by adopting the operator preconditioning framework in \cite{Mardal-Winther:2011}. 

After the Crank-Nicolson time discretization we need to solve the static equation
\algns{
\mathcal{B} ((\ub^{n+1}, \pf^{n+1}), ({\vb}, {\qf})) = \mathcal{F}^{n+1} ({\vb}, {\qf}), \qquad ({\vb}, {\qf}) \in \Vbh \times \Qh
}
where $\mathcal{F}^{n+1}$ is determined by the previous time step solution $(\ub^{n}, \pf^{n})$ and the right-hand side terms, and $\mathcal{B} : (\Vbh \times \Qh) \times (\Vbh \times \Qh) \ra \R$ is defined by 
\algns{
  \mathcal{B}((\vb, \qf), (\tilde{\vb}, \tilde{\qf})) &= a_{\ub} (\vb, \tilde{\vb}) - \LRp{ \alpha \qf, \div \tilde{\vb} }_{\Omega} - \LRp{ \alpha \tilde{\qf}, \div \vb }_{\Omega} - \LRp{ s_0 \qf, \tilde{\qf}}_{\Omega} - S\LRp{ \qf, \tilde{\qf}}  \\
  &\quad - \frac{\Delta t}2 a_p(\qf, \tilde{\qf}) .
}
Let  $\tilde{Q}_h$ be the space $\Qh$ with the norm 
\algns{
\| \qf \|_{\tilde{Q}_h}^2 = \| \qf^c \|_{0,\laminv}^2 + \| \qf^c \|_{0,s_0}^2 + \| \qf^0 \|_{0,\laminv}^2 + \| \qf^0 \|_{0,s_0}^2 + S(\qf^0, \tilde{\qf}^0) + \frac{\Delta t}2 \| \qf \|_h^2 .
}
Where $q^0, \tilde{q}^0$ are mean-value zero on $\Omega$. For simplicity let us define $\| (\vb, \qf) \|_{\mathcal{X}}$ by
\algns{
 \| (\vb, \qf) \|_{\mathcal{X}}^2 = \| \vb \|_{\Vbh}^2 + \| \qf \|_{\tilde{Q}_h}^2  .
}
We will prove that the linear operator $\mathcal{L}_{\mathcal{B}} : \Vbh \times \tilde{Q}_h \ra (\Vbh \times \tilde{Q}_h)^*$ defined by 
\algn{ \label{eq:L_B-definition}
\LRa{\mathcal{L}_{\mathcal{B}} (\vb, \qf), (\tilde{\vb}, \tilde{\qf})}_{\LRp{(\Vbh \times \tilde{Q}_h)^*, \Vbh \times \tilde{Q}_h}} := \mathcal{B} ((\vb, \qf), (\tilde{\vb}, \tilde{\qf})), 
}
is a bounded isomorphism from $\Vbh \times \tilde{Q}_h$ to the dual space $(\Vbh \times \tilde{Q}_h)^*$. Then, a parameter-robust preconditioner can be constructed as the inverse of the block diagonal operator 
\algn{ \label{eq:preconditioner} 
\mathcal{P} = 
\begin{pmatrix}
  \mathcal{P}_{\Vbh} & 0 \\ 0 & \mathcal{P}_{\tilde{Q}_h}
\end{pmatrix}
}
such that $\mathcal{P}_{\Vbh}:\Vbh \ra \Vbh^*$ and $\mathcal{P}_{\tilde{Q}_h}:\tilde{Q}_h \ra \tilde{Q}_h^*$ are the bounded linear operators defined by 
\algns{
  \LRa{\mathcal{P}_{\Vbh} \vb, \tilde{\vb} }_{(\Vbh^*, \Vbh)} &:= a_{\ub} (\vb, \tilde{\vb}), \\
  \LRa{\mathcal{P}_{\tilde{Q}_h} \qf, \tilde{\qf} }_{(\tilde{Q}_h^*, \tilde{Q}_h)} &:= \LRp{s_0 \qf^c, \tilde{\qf}^c }_{\Omega} + \LRp{\laminv \qf^c, \tilde{\qf}^c }_{\Omega} \\
  &\quad + \frac{\Delta t}2 \LRp{  \int_{\Omega} \kapb \nabla q^c \cdot \nabla q^c \, dx + \sum_{e \in \mc{E}_h^{\pd} } \gamma h_e^{-1} \LRa{ \jump{q^c}, \jump{\tilde{q}^c} }_e } \\
  &\quad + \LRp{s_0 \qf^0, \tilde{\qf}^0 }_{\Omega} + \LRp{\laminv \qf^0, \tilde{\qf}^0 }_{\Omega} + S(\qf^0, \tilde{\qf}^0) \\
  &\quad + \frac{\Delta t}2 \LRp{ \sum_{e \in \mc{E}_h^{0} } \gamma h_e^{-1-\beta} \LRa{\jump{q^0}, \jump{\tilde{\qf}^0} }_e  + \sum_{e \in \mc{E}_h^{\pd} } \gamma h_e^{-1} \LRa{ \jump{q^0}, \jump{\tilde{\qf}^0} }_e } . 
}
If $\mathcal{L}_{\mathcal{B}} : \Vbh \times \tilde{Q}_h \ra (\Vbh \times \tilde{Q}_h)^*$ is an isomorphism, then a preconditioner of the form \eqref{eq:preconditioner} is expected to give parameter-robust preconditioners based on norm equivalence (cf. \cite{Mardal-Winther:2011}). 

By the definition \eqref{eq:L_B-definition}, $\mathcal{L}_{\mathcal{B}} : \Vbh \times \tilde{Q}_h \ra (\Vbh \times \tilde{Q}_h)^*$ is a bounded isomorphism independent of $h, \dt, \lambda$ if $\mathcal{B}$ is a bounded bilinear form on $\Vbh \times \tilde{Q}_h$, and the following inf-sup condition holds: There exists $C>0$ independent of $h$, $\lambda$, $\Delta t$ such that 
\algn{ \label{eq:preconditioning-inf-sup}
	\inf_{(\tilde{\vb}, \tilde{\qf}) \in \Vbh \times \Qh} \sup_{(\vb, \qf) \in \Vbh \times \Qh} \frac{\mathcal{B}((\vb, \qf), (\tilde{\vb}, \tilde{\qf})) }{ \| (\bs{v}, q) \|_{\mathcal{X}} \| (\tilde{\vb}, \tilde{q}) \|_{\mathcal{X}}  } \ge C > 0 .
}
The main result in this section is proving the inf-sup condition \eqref{eq:preconditioning-inf-sup}. As a consequence, $\mathcal{L}_{\mathcal{B}}$ is an isomorphism, so preconditioners of the form \eqref{eq:preconditioner} give robust preconditioners. 

Before we prove this we state a necessary result which is proved in \cite{Lee-Ghattas}. 
\begin{lemma} \label{lemma:ap-inf-sup}
  For given $\tilde{\qf} = \tilde{\qf}^c + \tilde{\qf}^0$,  one can find $r^0 \in \Qh^0$ satisfying
  \algn{ 
   	\label{eq:r0-ineq0} & \sum_{e\in \mathcal{E}_h^{0}} \LRa{ h_e^{-1}  \jump{r^0}, \jump{r^0} }_e  \le C_r \sum_{e \in \mathcal{E}_h^0} \LRa{ h_e^{-1} \jump{\tilde{\qf}^0}, \jump{\tilde{\qf}^0} }_e , \\
   	\label{eq:r0-ineq1} &\sum_{e \in \mathcal{E}_h^{\pd}} \LRa{ h_e^{-1} \jump{r^0}, \jump{r^0} }_e +  \sum_{e\in \mathcal{E}_h^{0}} \LRa{ h_e^{-1-\beta}  \jump{r^0}, \jump{r^0} }_e \\
    \notag &\quad \le C_r \sum_{e \in \mathcal{E}_h^0} \LRa{ h_e^{-1-\beta} \jump{\tilde{\qf}^0},  \jump{\tilde{\qf}^0} }_e , \\
  	\label{eq:r0-ineq2} & \| r^0 \|_{0} \le C_r \| \tilde{\qf}^0 \|_0 ,
  }
  for some $C_r > 0$ independent of $h$, and $\qf = \tilde{\qf} + \delta r^0$ with sufficiently small (but independent of $h$) $\delta>0$ satisfies
  \algn{ \label{eq:eg-inf-sup}
    a_p(\qf, \tilde{\qf}) \ge C_{EG} \| \qf \|_h^2 
  }
  with $C_{EG}$ depending on the choice of $\delta$. In other words, the inf-sup condition
  \algns{
    \inf_{ \tilde{\qf} \in \Qh } \sup_{\qf \in \Qh} \frac{a_p(\qf, \tilde{\qf}) }{ \| q \|_{h} \| \tilde{q} \|_{h}  } \ge C_{EG}' > 0 
  }
  holds with $C_{EG}'$ independent of $h$.
\end{lemma}
We can now prove the main result.
\begin{theorem} \label{thm:preconditioning-inf-sup}
  The inf-sup condition \eqref{eq:preconditioning-inf-sup} holds.
\end{theorem}
\begin{proof}[Proof of Theorem~\ref{thm:preconditioning-inf-sup}]
For given $(\tilde{\vb}, \tilde{\qf}) \in \Vbh \times \Qh$ it suffices to show that there exists $(\vb, \qf) \in \Vbh \times \Qh$ such that 
\algns{
  \| (\vb, \qf) \|_{\mathcal{X}} &\le C_5  \| (\tilde{\vb}, \tilde{q}) \|_{\mathcal{X}} , \\
  \mathcal{B}((\vb, \qf), (\tilde{\vb}, \tilde{\qf})) &\ge C_6  \| (\tilde{\vb}, \tilde{q}) \|_{\mathcal{X}}^2
}
for some constants $C_5, C_6 >0$ independent of $h$, $\lambda$, $\Delta t$. 

Suppose that $(0,0)\not = (\tilde{\vb}, \tilde{\qf}) \in \Vbh \times \Qh$ is given. 
By Lemma~\ref{lemma:ap-inf-sup}, there exists $r^0 \in \Qh^0$ satisfying \eqref{eq:r0-ineq1}, \eqref{eq:r0-ineq2} such that $a_p(\tilde{\qf} + \delta_0 r^0, \tilde{\qf}) \ge C_3 \| \tilde{\qf} \|_h^2$ and $\| \tilde{\qf} + \delta_0 r^0 \|_h \le C_4 \| \tilde{\qf} \|_h$ with sufficiently small $\delta_0>0$ for $C_3, C_4>0$ independent of $h$, $\lambda$, $\Delta t$. 
By the same argument in the proof of Lemma~\ref{lemma:aux-inf-sup}, there exists $\wb_1 \in \Vbh$ such that
$\| \wb_1 \|_{\Vbh} = \| \tilde{\qf} \|_{0, \laminv}$ and 
%
\algn{ \label{eq:w-weak-inf-sup3}
  {\LRp{\div \wb_1, \tilde{q}}_{\Omega} } \ge C_1' \| \tilde{q} \|_{0,\laminv}^2 - C_2' S(\tilde{q}, \tilde{q}) 
}
for all $\tilde{\qf} \in \Qh$. Since $\div \Vbh = \Qh^0$, there exists $\wb_2 \in \Vbh$ such that $\div \wb_2 = \lambda^{-1} \tilde{\qf}^0$ and $\| \wb_2 \|_{\Vbh} \le C \| \tilde{\qf}^0 \|_{0,\laminv}$. 

Take $(\vb, q)$ as $\vb = \tilde{\vb} + \delta_1 \wb_1$ and $q = - \tilde{q} - \delta_0 r^0$ with $\delta_0, \delta_1>0$ which will be determined later. Then, 
\algn{ \label{eq:B-intm2}
  \mathcal{B}((\vb, q), (\tilde{\vb}, \tilde{q})) &= \| \tilde{\vb} \|_{\Vbh}^2 + \delta_1 a_{\ub} (\wb_1, \tilde{\vb}) + \delta_1 \LRp{ \alpha \tilde{q}, \div \wb_1 }_{\Omega} \\
  &\quad + \delta_0 \LRp{ \alpha r^0, \div \tilde{\vb} }_{\Omega}  + \| \tilde{q} \|_{0,s_0}^2 + \delta_0 \LRp{ s_0 r^0, \tilde{q}}_{\Omega} \\
  &\quad + S(\tilde{q}, \tilde{q}) + \delta_0 S(r^0, \tilde{q}) + \frac{\Delta t}2 a_p \LRp{\tilde{\qf} + \delta_0 r^0, \qf }.
}
By \eqref{eq:w-weak-inf-sup3}, the parallelogram law, 
\algn{
2 \LRp{\|\tilde{q}\|_{0, \lambda^{-1}}^2 + \|\tilde{q}^0\|_{0, \lambda^{-1}}^2} = \|\tilde{q} + \tilde{q}^0\|_{0, \lambda^{-1}}^2 + \|\tilde{q} - \tilde{q}^0\|_{0, \lambda^{-1}}^2 \ge \|\tilde{q}^c\|_{0, \lambda^{-1}}^2
}
and Lemma~\ref{lemma:discrete-poincare}, we have
\algn{ 
  \notag \delta_1 \LRp{ \alpha \tilde{q}, \div \wb_1}_{\Omega} &\ge \delta_1 C_1'' \| \tilde{q} \|_{0,\laminv}^2 - \delta_1 C_2'' S(\tilde{q}, \tilde{q}) \\
 \notag  &\ge \frac{\delta_1 C_1''}2 \| \tilde{q}^c \|_{0,\laminv}^2 - \delta_1 C_1'' \| \tilde{q}^0 \|_{0,\laminv}^2 - \delta_1 C_2'' S(\tilde{q}, \tilde{q}) \\
 \label{eq:aux-estm1} &\ge  \frac{\delta_1 C_1''}2 \| \tilde{q}^c \|_{0,\laminv}^2 - \delta_1 \LRp{ C_P C_1'' + C_2'' } S(\tilde{q}, \tilde{q}) 
}
with $C_1'', C_2''$ which depend on the lower bound of $\alpha$ but are independent of $h$, $\lambda$, $\Delta t$. By Young's inequality, 
\algn{ \label{eq:aux-estm2}
  \delta_1 | a_{\ub} (\wb_1, \tilde{\vb}) | &\le \frac 14 \| \tilde{\vb} \|_{\Vbh}^2 + \delta_1^2 \| \wb_1 \|_{\Vbh}^2 = \frac 14 \| \tilde{\vb} \|_{\Vbh}^2 + \delta_1^2 \| \tilde{q} \|_{0,\laminv}^2 .
}
%
%
%
Moreover, we can obtain
\algn{ 
  \notag \delta_0 | \LRp{\alpha r^0, \div \tilde{\vb} }_{\Omega} | &\le \delta_0 \| r^0 \|_{0,\laminv} \| \vb \|_{\Vbh} \le C_r \delta_0 \| \tilde{\qf}^0 \|_{0,\laminv} \| \vb \|_{\Vbh} \\
  \label{eq:aux-estm3} &\le C_r^2 \delta_0^2  \| \tilde{\qf}^0 \|_{0,\laminv}^2 + \frac 14 \| \vb \|_{\Vbh}^2 , \\
  \notag \delta_0 | \LRp{ s_0 r^0, \tilde{\qf}}_{\Omega} | &\le C_r^2 \delta_0^2 \| \tilde{\qf}^0 \|_{0,s_0}^2 + \frac 12 \| \tilde{\qf} \|_{0,s_0}^2 \\
   \label{eq:aux-estm4} &\le C_r^2 C_P^2 \delta_0^2 S(\tilde{\qf}, \tilde{\qf}) + \frac 12 \| \tilde{\qf} \|_{0,s_0}^2 , \\
  \label{eq:aux-estm5} \delta_0 | S(r^0, \tilde{\qf}) | &\le C_r \delta_0 S(\tilde{\qf}, \tilde{\qf}) .
}
If we use \eqref{eq:aux-estm1}--\eqref{eq:aux-estm5} to \eqref{eq:B-intm2}, then 
\algns{ 
  \mathcal{B}((\vb, q), (\tilde{\vb}, \tilde{q})) &\ge \frac 12 \| \tilde{\vb} \|_{\Vbh}^2 - \delta_1^2 \| \tilde{\qf} \|_{0,\laminv}^2 + \frac{\delta_1 C_1''}2 \| \tilde{\qf}^c \|_{0,\laminv}^2 \\
  &\quad - \delta_1 \LRp{C_P C_1'' + C_2'' } S(\tilde{\qf}, \tilde{\qf}) \\
  &\quad - C_r^2 \delta_0^2 \| \tilde{\qf}^0 \|_{0,\laminv}^2 - C_r^2 C_P^2 \delta_0^2 \| \tilde{\qf}^0 \|_{0,\laminv}^2 + \frac 12 \| \tilde{\qf} \|_{0,s_0}^2  \\
  &\quad + S(\tilde{q}, \tilde{q}) - C_r \delta_0 S(\tilde{\qf}, \tilde{\qf}) + \frac{C_3 \Delta t}2 \| \tilde{\qf} \|_h^2.
}
Since $\| \tilde{\qf}^c \|_{0,s_0}^2 = \| \tilde{\qf} + (- \tilde{\qf}^0) \|_{0,s_0}^2 \le 2(\| \tilde{\qf}\|_{0,s_0}^2 + \| \tilde{\qf}^0 \|_{0,s_0}^2)$, we have 
\algns{
  \frac 12 \| \tilde{\qf} \|_{0,s_0}^2 \ge \delta_0 \| \tilde{\qf} \|_{0,s_0}^2 \ge \frac{\delta_0}2 \| \tilde{\qf}^c \|_{0,s_0}^2 - \delta_0 \| \tilde{\qf}^0 \|_{0,s_0}^2 
}
if $\delta_0 \le \frac 12$. If we apply this to the previous inequality assuming that $\delta_0 \le \frac 12$, then we have 
\algns{ 
  \mathcal{B}((\vb, q), (\tilde{\vb}, \tilde{q})) &\ge \frac 12 \| \tilde{\vb} \|_{\Vbh}^2 - \delta_1^2 \| \tilde{\qf} \|_{0,\laminv}^2 + \frac{\delta_1 C_1''}2 \| \tilde{\qf}^c \|_{0,\laminv}^2 \\
  &\quad - \delta_1 \LRp{\frac{C_P C_1''}2 + C_2'' } S(\tilde{\qf}, \tilde{\qf}) \\
  &\quad - C_r^2 \delta_0^2 \| \tilde{\qf}^0 \|_{0,\laminv}^2 - C_r^2 C_P^2 \delta_0^2 \| \tilde{\qf}^0 \|_{0,\laminv}^2 + \frac{\delta_0}2 \| \tilde{\qf}^c \|_{0,s_0}^2 - \delta_0 \| \tilde{\qf}^0 \|_{0,s_0}^2  \\
  &\quad + S(\tilde{q}, \tilde{q}) - C_r \delta_0 S(\tilde{\qf}, \tilde{\qf}) + \frac{C_3 \Delta t}2 \| \tilde{\qf} \|_h^2.
}
If we choose a sufficiently small $\delta_0$, $\delta_1$ satisfying
\algns{
  \delta_1 \LRp{ C_P C_1'' + C_2''} + C_r \delta_0 \le \frac 14, \\
  \max \{ \delta_1^2 + C_r^2 \delta_0^2  + C_r^2 C_P^2 \delta_0^2, \delta_0 \} C_P^2 \le \frac 14, 
}
we obtain 
\mltlns{ 
  \mathcal{B}((\vb, q), (\tilde{\vb}, \tilde{q})) 
  \\
  \ge \frac 12 \| \tilde{\vb} \|_{\Vbh}^2 + \frac{\delta_1 C_1''}2 \| \tilde{\qf}^c \|_{0,\laminv}^2 + \frac{\delta_0}2 \| \tilde{\qf}^c \|_{0,s_0}^2 + \frac 12 S(\tilde{\qf}, \tilde{\qf})  + \frac{C_3 \Delta t}2 \| \tilde{\qf} \|_h^2.
}
Since $\| \tilde{\qf}^0 \|_{0,\laminv} + \| \tilde{\qf}^0 \|_{0,s_0} \le C_P S(\tilde{\qf}^0, \tilde{\qf}^0)^{1/2}$, the above inequality gives
\algns{
    \mathcal{B}((\vb, q), (\tilde{\vb}, \tilde{q})) \ge C \| (\tilde{\vb}, \tilde{\qf}) \|_{\mathcal{X}}^2
}
with $C>0$ depending on $C_1''$, $C_2$, $C_3$, $\delta_0$, $\delta_1$.

On the other hand, by the definition of $(\vb, \qf)$, it is easy to check that $\| (\vb, \qf) \|_{\mathcal{X}} \le C \| (\tilde{\vb}, \tilde{\qf}) \|_{\mathcal{X}}$ with $C>0$ depending on $\delta_0$, $\delta_1$. Therefore, the conclusion follows.
\end{proof}

\begin{table}[h]
\tiny{
	\begin{center}
		\begin{tabular}{c|c|c||cc cc cc cc cc} 
		\hline
			\multicolumn{3}{c||}{} & \multicolumn{2}{c}{$ \| u - u_h \|_{a_{\ub}} / \| u \|_{a_{\ub}} $} & \multicolumn{2}{c}{$ \| u - u_h \|_1 / \| u \|_1 $} & \multicolumn{2}{c}{$ \| u - u_h \|_0 / \| u \|_0$} & \multicolumn{2}{c}{$ \| p - p_h \|_1 / \| p \|_1 $} & \multicolumn{2}{c}{$ \| p - p_h \|_0 / \| p \|_0$}\\
			 $\beta$ & $\nu$ & $N$ & error & rate & error & rate & error & rate & error & rate & error & rate \\ \hline \hline 
			\multirow{8}{*}{1} & \multirow{4}{*}{0.3} & $8$  & 1.1861e-01  &  --  &  1.0361e-01  &  --  &  1.0997e-02 & -- & 1.0178e-01 & -- & 9.9130e-03 & --\\ 
			 & & $16$ & 5.9533e-02  &  0.99  &  5.1796e-02  &  1.00  &  2.8000e-03 & 1.97 & 5.1074e-02 & 0.99 & 2.5883e-03 & 1.94\\ 
			 & & $32$ & 2.9788e-02  &  1.00  &  2.5882e-02  &  1.00  &  7.0630e-04 & 1.99 & 2.5545e-02 & 1.0 & 6.9010e-04 & 1.91\\ 
			 & & $64$ & 1.4895e-02  &  1.00  &  1.2935e-02  &  1.00  &  1.7740e-04 & 1.99 & 1.2767e-02 & 1.0 & 1.7800e-04 & 1.95\\ 
			\cline{3-13} 
            & \multirow{4}{*}{0.499} & $8$  & 1.2029e-01  &  --  &  1.0223e-01  &  --  &  1.1052e-02 & -- & 1.0180e-01 & -- & 1.0211e-02 & --\\
			 & & $16$ & 6.0466e-02  &  0.99  &  5.1133e-02  &  1.00  &  2.8143e-03 & 1.97 & 5.1080e-02 & 0.99 & 2.6653e-03 & 1.94\\ 
			 & & $32$ & 3.0272e-02  &  1.00  &  2.5555e-02  &  1.00  &  7.1040e-04 & 1.99 & 2.5542e-02 & 1.0 & 6.8090e-04 & 1.97\\ 
			 & & $64$ & 1.5141e-02  &  1.00  &  1.2772e-02  &  1.00  &  1.7850e-04 & 1.99 & 1.2767e-02 & 1.0 & 1.7460e-04 & 1.96\\ 
			 \cline{1-13} 
			\multirow{8}{*}{2} & \multirow{4}{*}{0.3} & $8$  & 1.1861e-01  &  --  &  1.0361e-01  &  --  &  1.0997e-02 & -- & 1.0179e-01 & -- & 9.9147e-03 & --\\ 
			 & & $16$ & 5.9533e-02  &  0.99  &  5.1796e-02  &  1.00  &  2.8001e-03 & 1.97 & 5.1082e-02 & 0.99 & 2.5884e-03 & 1.94\\ 
			 & & $32$ & 2.9788e-02  &  1.00  &  2.5882e-02  &  1.00  &  7.0630e-04 & 1.99 & 2.5546e-02 & 1.0 & 6.9010e-04 & 1.91\\ 
			 & & $64$ & 1.4895e-02  &  1.00  &  1.2935e-02  &  1.00  &  1.7740e-04 & 1.99 & 1.2767e-02 & 1.0 & 1.7800e-04 & 1.95\\ 
			\cline{3-13} 
            & \multirow{4}{*}{0.499} & $8$  & 1.2029e-01  &  --  &  1.0223e-01  &  --  &  1.1052e-02 & -- & 1.0181e-01 & -- & 1.0213e-02 & --\\
			 & & $16$ & 6.0466e-02  &  0.99  &  5.1133e-02  &  1.00  &  2.8143e-03 & 1.97 & 5.1080e-02 & 0.99 & 2.6654e-03 & 1.94\\ 
			 & & $32$ & 3.0272e-02  &  1.00  &  2.5545e-02  &  1.00  &  7.1040e-04 & 1.99 & 2.5543e-02 & 1.0 & 6.8090e-04 & 1.97\\ 
			 & & $64$ & 1.5141e-02  &  1.00  &  1.2772e-02  &  1.00  &  1.7850e-04 & 1.99 & 1.2767e-02 & 1.0 & 1.7460e-04 & 1.96\\ 
			 \cline{1-13} 			 
		\end{tabular} 
		\caption{ The convergence results of errors with the manufactured solution \eqref{eq:manufactured-solution} for $\beta = 1,2$ and the Crank--Nicolson scheme with $\Delta t= 1/N$. 
}
	\label{table:conv-quadratic}
	\end{center}
}
 \end{table}

\begin{table}[h]
	\begin{center}
		\begin{tabular}{c|c|c||c c c c} 
			\multicolumn{3}{c||}{ } & \multicolumn{4}{c}{$N~(DOFs)$} \\ 
			$\beta$ &$\Delta t$ &$\text{$\nu$}$ &$8~(3201)$ &$16~(12545)$ &$32~(49665)$ &$64~(197633)$ \\ 
			\hline 
			\multirow{6}{*}{$0$} &\multirow{2}{*}{$10^{-1}$} & $ 0.3$ & $148$ & $195$ & $243$ & $307$ \\ 
			& &$ 0.499$ & $185$ & $192$ & $185$ & $163$ \\
			\cline{3-7} 
			& \multirow{2}{*}{$10^{-2}$} & $ 0.3$ & $152$ & $210$ & $276$ & $343$ \\
			& &$ 0.499$ & $214$ & $254$ & $252$ & $231$ \\
			\cline{3-7} 
			& \multirow{2}{*}{$10^{-3}$} & $ 0.3$ & $152$ & $218$ & $295$ & $391$  \\
			& &$ 0.499$ & $228$ & $313$ & $333$ & $336$ \\
			\cline{1-7}
			\multirow{6}{*}{$1$} &\multirow{2}{*}{$10^{-1}$} & $ 0.3$ & $58$ & $57$ & $57$ & $57$ \\ 
			& &$ 0.499$ & $91$ & $92$ & $93$ & $98$ \\
			\cline{3-7} 
			& \multirow{2}{*}{$10^{-2}$} & $ 0.3$ & $60$ & $60$ & $60$ & $58$ \\
			& &$ 0.499$ & $94$ & $97$ & $102$ & $102$ \\
			\cline{3-7} 
			& \multirow{2}{*}{$10^{-3}$} & $ 0.3$ & $61$ & $62$ & $59$ & $61$ \\
			& &$ 0.499$ & $97$ & $100$ & $105$ & $106$ \\
			\cline{1-7}
			\multirow{6}{*}{$2$} &\multirow{2}{*}{$10^{-1}$} & $ 0.3$ & $39$ & $36$ & $35$ & $33$ \\ 
			& &$ 0.499$ & $62$ & $61$ & $63$ & $62$ \\
			\cline{3-7} 
			& \multirow{2}{*}{$10^{-2}$} & $ 0.3$ & $39$ & $38$ & $36$ & $35$ \\
			& &$ 0.499$ & $63$ & $63$ & $64$ & $64$ \\
			\cline{3-7} 
			& \multirow{2}{*}{$10^{-3}$} & $ 0.3$ & $40$ & $39$ & $38$ & $36$ \\
			& &$ 0.499$ & $65$ & $64$ & $66$ & $65$ \\
			\hline 
	\label{table:scalar-kappa-precond} 
		\end{tabular} 
		\caption{The numbers of iteration for $\beta = 0,1,2$ of the MinRes algorithm with abstract preconditioners of the form \eqref{eq:preconditioner}. See Section~\ref{sec:numerics} for more details on implementation. }
	\end{center} 
\end{table}

{

\section{Numerical results}
\label{sec:numerics}

In this section we present numerical results which support our theoretical findings on error analysis and preconditioning. All numerical experiments are done with the Firedrake package \cite{Firedrake:main}. 

In our numerical experiments we set $\Omega = [0,1]\times [0,1]$ with $\Gamma_{d} = \{0\} \times [0,1]$ and $\Gamma_p = \partial \Omega$. We use the structured meshes obtained by dividing $\Omega$ into $N \times N$ subsquares and the dividing each subsquare into two triangles. For convergence rate tests we refine the meshes with $N=8, 16, 32, 64$.  

To check convergence rates of the errors of numerical and exact solutions, we consider a manufactured solution with 
\begin{subequations}
\label{eq:manufactured-solution}
\algn{
	\ub &= \LRp{- \pi x^2 \LRp{1-x}^2 \sin^2 \LRp{\pi y} \cos \LRp{2t}, - \pi y^2 \LRp{1-y}^2 \sin^2 \LRp{\pi x} \cos \LRp{2t}} ,
	\\ 
	\pf &= x \LRp{1-x} y \LRp{1-y} \cos \LRp{t} .
}
\end{subequations}
Note that $\div \bs{u} \not = 0$, so the $L^2$ norm of $\bs{f}$ may grow unboundedly as $\lambda \rightarrow +\infty$. In order to illustrate convergence rates of the errors independent of $\lambda$, we compute the relative errors measured by
\gats{
	\| \bs{u} - \bs{u}_h \|_{a_{\ub}} / \| \ub \|_{a_{\ub}}, \quad  \| \ub - \ub_h \|_1 / \| \ub \|_1, \quad  \| \ub - \ub_h \|_0 / \| \ub \|_0,
	\\
	\| p - p_h \|_1 / \| p \|_1, \quad \| p - p_h \|_0 / \| p \|_0 .
}
In Table~\ref{table:conv-quadratic} we present the results for $\beta = 1,2$ cases with the Crank--Nicolson time discretization with time step size $\Delta t = 1/N$. For all relative error quantities and for $\nu = 0.3, 0.499$ we observe that optimal convergence rates are obtained. 

For preconditioning we also present preconditioning test results. We compare the results for all combinations of $\beta=0,1,2$, $\nu = 0.3, 0.499$, $\Delta t= 10^{-1},10^{-2},10^{-3}$.  We use preconditioners which have the abstract form \eqref{eq:preconditioner}. For implementation of $\mathcal{P}_{\bs{V}_h}$ we use the geometric multigrid preconditioner developed in \cite{Aznaran-Farrell-Kirby:2021}. For implementation of $\mathcal{P}_{\tilde{Q}_h}$ we use the preconditioners proposed in \cite{Lee-Lee-Wheeler-16,Lee-Ghattas} combining algebraic multigrid preconditioners for $Q_h^c$ and the block Jacobi preconditioner for $Q_h^0$. Numbers of iterations are given in Table~\ref{table:scalar-kappa-precond}. One can observe that $\beta=0$ case, the original enriched Galerkin method for $p$, shows strong dependence on mesh refinement and $\nu \approx 0.5$. For $\beta=1$, the dependence on mesh refinement disappears but the numbers of iteration still show some dependence for $\nu \approx 0.5$. For $\beta=2$, the numbers of iteration are $\approx 65$. We remark that these numbers are usually regarded as optimal iteration numbers in parameter-robust preconditioning with MinRes algorithm \cite{LeeEtAl2017}. The iteration numbers become much smaller in practice if GmRes algorithm is used but the norm equivalence discussion in Section~\ref{sec:preconditioning} is not valid for GmRes, so we do not include GmRes in this paper.

\section{Conclusion}
\label{sec:conclusion}
In this paper we developed a new discretization for poroelasticity. We use the primal formulations for the elastic displacement and for the pore pressure minimizing the number of variables. 
In our discretization with nonconforming and the interior over-stabilized enriched Galerkin methods, numerical solutions are robust for nearly incompressible elasticity parameters, are locally mass conservative. Furthermore, we showed a theoretical discussion how to construct robust preconditioinors and the robust preconditioner construction method was verified by numerical experiments. 

\section*{Statements and Declarations}
\noindent{\bf Funding} Jeonghun J. Lee is partially supported by Baylor's Undergraduate Research Council grant, Baylor's start-up grant, and National Science Foundation (DMS-2110781) 

\noindent{\bf Conflict of Interest} The authors declared that they have no conflict of interest.

\providecommand{\bysame}{\leavevmode\hbox to3em{\hrulefill}\thinspace}
\providecommand{\MR}{\relax\ifhmode\unskip\space\fi MR }
\providecommand{\MRhref}[2]{%
  \href{http://www.ams.org/mathscinet-getitem?mr=#1}{#2}
}
\providecommand{\href}[2]{#2}


\end{document}